\documentclass[11pt]{amsart}
\usepackage{comment}
\usepackage{amssymb,amsmath,amsthm,mathrsfs,amstext,amssymb}
\usepackage{dsfont}
\usepackage{tabularx}
\usepackage{array}
\usepackage[square,sort,comma,numbers]{natbib}
\usepackage{url} 
\usepackage{hyperref}
\usepackage{accents}
\usepackage[shortlabels]{enumitem}
\usepackage{tikz-cd}

\newtheorem{theorem}{Theorem}
\numberwithin{theorem}{section}
\newtheorem*{theorem*}{Theorem}

\newtheorem*{maintheorem*}{Main Theorem}
\newtheorem{lemma}[theorem]{Lemma}
\newtheorem*{lemma*}{Lemma}

\newtheorem*{fact*}{Fact}
\newtheorem{corollary}[theorem]{Corollary}
\newtheorem*{corollary*}{Corollary}
\newtheorem{proposition}[theorem]{Proposition}
\newtheorem*{proposition*}{Proposition}

\theoremstyle{definition}

\newtheorem{definition}[theorem]{Definition}
\newtheorem*{definition*}{Definition}
\newtheorem{remark}[theorem]{Remark}
\newtheorem*{remark*}{Remark}

\newtheorem*{observation*}{Observation}
\newtheorem{question}[theorem]{Question}
\newtheorem*{question*}{Question}

\newcommand{\calE}{{\mathcal{E}}}
\newcommand{\calF}{{\mathcal{F}}}
\newcommand{\calG}{{\mathcal{G}}}

\newcommand{\calI}{{\mathcal{I}}}

\newcommand{\calP}{{\mathcal{P}}}
\newcommand{\calQ}{{\mathcal{Q}}}
\newcommand{\calR}{{\mathcal{R}}}

\newcommand{\calV}{{\mathcal{V}}}

\newcommand{\bbE}{{\mathbb{E}}}

\newcommand{\fra}{{\mathfrak{a}}}

\newcommand{\frc}{{\mathfrak{c}}}

\DeclareMathOperator{\dom}{dom}
\DeclareMathOperator{\ran}{ran}
\DeclareMathOperator{\restr}{\upharpoonright}

\DeclareMathOperator{\partialto}{{\overset{\text{part}}{\to}}}
\DeclareMathOperator{\length}{lh}

\newcommand{\simpleset}[1]{{\{{#1}\}}}

\newcommand{\set}[2]{{\{ {#1} \mid {#2} \}}}
\newcommand{\seq}[2]{{\langle {#1} \mid {#2} \rangle}}
\newcommand{\enumoperator}[4]{{{#1}_{#2}{#4} \dots {#4} {#1}_{#3}}}
\newcommand{\enum}[3]{{{#1}_{#2}, \dots, {#1}_{#3}}}
\newcommand{\enumzero}[2]{{\enum{#1}{0}{#2}}}
\newcommand{\enumone}[2]{{\enum{#1}{1}{#2}}}
\newcommand{\sizeof}[1]{{|{#1}|}}

\DeclareMathOperator{\extends}{{\mathbin{\leq}}}

\newcommand{\finseq}[1]{{^{<\omega}{#1}}}
\newcommand{\infseq}[1]{{^{\omega}{#1}}}
\newcommand{\finsubset}[1]{{[#1]^{<\omega}}}
\newcommand{\infsubset}[1]{{[#1]^{\omega}}}
\DeclareMathOperator{\Fin}{Fin}
\DeclareMathOperator{\Seq}{Seq}

\DeclareMathOperator{\spec}{spec}
\DeclareMathOperator{\aE}{\fra_{\text{\normalfont e}}}
\DeclareMathOperator{\av}{\fra_{\text{\normalfont v}}}

\newcommand{\functionspace}[2]{{^{#1}{#2}}}
\DeclareMathOperator{\bairespace}{{^\omega \omega}}

\title{Productivity of maximal eventually different families}

\author{Lukas Schembecker}
\address{Department of Mathematics, University of Hamburg, Bundesstraße 55, 20146 Hamburg, Germany}
\email{lukas.schembecker@uni-hamburg.de}

\begin{document}	
	\maketitle
	
	\begin{abstract}
		A maximal eventually different family is called $n$-productive if the product family $\calF^n$ is still maximal.
		We construct closed $n$-productive families separating these strengthenings of maximality at every $n \geq 1$.
		Furthermore, we show how to force and construct an even stronger type of $\calI_0$-productive family and discuss the relation of productivity to Van Douwen families.
	\end{abstract}

	\section{Introduction} \label{SEC_Introduction}
	
	Given two maximal eventually different (med) families of functions $\calF, \calG \subseteq \bairespace$, their product $\calF \times \calG$ (see Definition~\ref{DEF_Product}) is another eventually different (ed) family of functions, now taking values in the still countable set $\omega \times \omega$.
	An immediate natural question is whether and when such products are again maximal (see Remark~\ref{REM_ProductiveEquivalence}).
	Such products turn out to be very useful for the study of spectra of Van Douwen families (see Definition~\ref{DEF_VD}).
	Using such products in a gluing argument the author proved in \cite{Schembecker_2026} that the spectrum of Van Douwen families is closed under singular limits.
	In fact, Van Douwen families $\calV$ may be characterized as exactly those eventually different families such that $\calV \times \calF$ is maximal for any maximal eventually different family $\calF$ (see Proposition~\ref{PROP_VDEquivalence}).
	
	In particular Van Douwen families have the property that any finite product of $\calV$ with itself is still maximal.
	For $n \geq 1$ we call a med family $\calF$ $n$-productive if its $n$-fold product $\calF^n$ is still maximal.
	If this holds for every $n \geq 1$ we call $\calF$ finitely productive.
	In \cite{Raghavan_2010} Raghavan introduced the associated ideal $\calI_0(\calF)$ for an ed family $\calF$ (see Definition~\ref{DEF_AssociatedIdeal}).
	It is easy to see that, considering the finite products of $\calF$, we even get an associated increasing chain of ideals
	\[
		\calI_0(\calF) \subseteq \calI_0(\calF^2) \subseteq \calI_0(\calF^3) \subseteq \dots
	\]
	In this context Van Douwen families satisfy the even stronger property of $\calI_0(\calF) = \calI_0(\calF^n)$ for all $n \geq 1$, which we call finitely $\calI_0$-productive (see Definition~\ref{DEF_I0Productive}).
	
	While Van Douwen families provide natural examples of productive families, they are quite a strong strengthening of maximality, for example in terms of complexity, med families may be Borel \cite{HorowitzShelah_2024}, even closed \cite{Schrittesser_2016}, whereas Van Douwen families cannot even be analytic \cite{Raghavan_2010}.
	Productivity and $\calI_0$-productivity therefore give natural intermediate strengthenings between classical maximality and Van Douwenness.
	Such strengthenings of maximality below Van Douwenness are also interesting for another reason:
	It is an open question whether the minimal size of a med family $\aE$ may be consistently smaller than the minimal size of a Van Douwen family $\av$.
	Since tightness implies Van Douwenness, a forcing construction preserving some $\calF$ while destroying Van Douwen families would likely need another strengthening of maximality for $\calF$ weaker than Van Douwenness.
	
	In the next section we will introduce all notions we will consider in the paper, in particular productivity and $\calI_0$-productivity.
	We further prove some easy observations and reformulations of productivity and prove that Van Douwen families may be characterized precisely as all \textquoteleft universally\textquoteright\ productive families (see Proposition~\ref{PROP_VDEquivalence}).
	
	In the third section we will separate the notion of $\calI_0$-productivity from Van Douwenness.
	It turns out tallness of the associated ideal $\calI_0(\calF)$ precisely characterizes when a med family $\calF$ cannot be restricted to a Van Douwen family (see Proposition~\ref{PROP_RestrictionToVD}); in other words $\calF$ is nowhere Van Douwen.
	We then prove that under {\sf MA}($\sigma$-centered) any proper ideal $\calI$ containing $\Fin$ may be realized as the associated ideal of a finitely $\calI_0$-productive family (see Theorem~\ref{THM_ManyI0Productive}):
	\begin{theorem*}
		Assume {\sf MA}($\sigma$-centered). Let $\calI$ be any proper ideal containing $\Fin$.
		Then there is a finitely $\calI_0$-productive med family $\calF$ with $\calI_0(\calF) = \calI$.
	\end{theorem*}

	This generalizes the author's realization of any proper $\calI$ containing $\Fin$ as associated ideal of some med family under {\sf CH} \cite{Schembecker_2026}.
	Further, we use a more forcing-centered argument here, which allows us to relax the assumption to {\sf MA}($\sigma$-centered).
	In the last two sections we will take a closer look at the definability of productive families.
	It turns out that while Van Douwen families cannot be analytic, we can have Borel finitely productive med families.
	Even better, we may separate the entire hierarchy of productivity by Borel witnesses (see Theorems~\ref{THM_BorelNProductive}~and~\ref{THM_BorelFinitelyProductive}).
	
	\begin{theorem*}
		For every $n \geq 1 $ there is a Borel med family $\calE_n$ which is $n$-productive, but not $(n+1)$-productive and there is a Borel finitely productive med family $\calE_{<\omega}$.
	\end{theorem*}

	For $n = 1$, the construction corresponds to the construction of the Borel med family $\calE$ by Horowitz and Shelah \cite{HorowitzShelah_2024}.
	We give an overview of their construction roughly following the presentation by Schrittesser in \cite{Schrittesser_2016} and show that the original family $\calE$ is not even $2$-productive (see Lemma~\ref{LEM_HSNotSquare}).
	Then we give a precise argument how to generalize the construction to obtain a Borel $2$-productive, but not $3$-productive med family.
	The proof describes all ideas necessary for the subsequent generalization of the construction to any $n \geq 1$ (see Theorem~\ref{THM_BorelNProductive}) and finally a Borel finitely productive family (see Theorem~\ref{THM_BorelFinitelyProductive}).
	
	In the last section we lower the complexity of the Borel families of the previous section to obtain closed families with the same finite productivity properties (see Corollary~\ref{COR_ClosedProductive}) and show that the family $\calE_{<\omega}$ also separates productivity from $\calI_0$-productivity (see Lemma~\ref{LEM_StrictlyIncreasingI0}).
	We end with some questions regarding the associated chain and the definability of $\calI_0$-productive families.

	\section{Definitions and notations} \label{SEC_Definitions}

	\begin{definition} \label{DEF_Product}
		Let $\enumone{f}{n}$ be a finite tuple of functions with the same domain.
		Then we denote with $\enumoperator{f}{1}{n}{\times}$ the product of functions (with the same domain) defined by:
		\[
			(\enumoperator{f}{1}{n}{\times})(a) := (f_1(a), \dots, f_n(a))
		\]
		For a finite tuple of families of functions $\enumone{\calF}{n}$ with the same domain we write
		\[
			\enumoperator{\calF}{1}{n}{\times} := \set{\enumoperator{f}{1}{n}{\times}}{f_1 \in \calF_1, \dots, f_n \in \calF_n}.
		\]
		Further, we abbreviate $\smash{\underset{n \text{ times}}{\calF \times \dots \times \calF}}$ with $\calF^n$.
	\end{definition}

	\begin{definition} \label{DEF_ED}
		Let $A, B$ be countable sets and $f,g \in \functionspace{A}{B}$.
		We will always denote their equality set by $E(f,g) := \set{a \in A}{f(a) = g(a)}$.
		We say that $f$ and $g$ are eventually different iff $E(f,g)$ is finite.
		A family $\calF \subseteq \functionspace{A}{B}$ is called an eventually different (ed) family (of functions) iff all its members are pairwise eventually different.
		If $\calF$ is further a maximal with respect to inclusion among such families, then we say $\calF$ is a maximal eventually different (med) family (of functions).
	\end{definition}

	\begin{definition} \label{DEF_SpecED}
		We define the spectrum and cardinal invariant associated to maximal eventually different families as:
		\begin{align*}
			\spec(\aE) &:= \set{\sizeof{\calF}}{\calF \subseteq \bairespace \text{ is med}},\\
			\aE &:= \min \spec(\aE).
		\end{align*}
	\end{definition}

	Usually, we will have $A = B = \omega$, but we will also naturally consider $A \in \infsubset{\omega}$ and $A = B = \omega^n$.
	The subsequent definitions will be defined for $A = B = \omega$, but this will also implicitly give definitions for any countable $A$ and $B$. 
	For example, there would be no difference considering any other countable domain and range than $\omega$ in the definition of $\spec(\aE)$ and $\aE$.
	
	\begin{definition} \label{DEF_VD}
		Let $\calF$ be med.
		Then $\calF$ is called Van Douwen iff for every $A \in \infsubset{\omega}$ also
		\[
			\calF \restr A := \set{f \restr A}{f \in \calF}
		\]
		is a maximal eventually different family of functions.
	\end{definition}
	
	Equivalently, $\calF$ is also maximal with respect to infinite partial functions, i.e.\ $\calF$ is Van Douwen iff for all infinite partial functions $p:A \to \omega$, there is some $f \in \calF$ with $f =^\infty p$.
	
	\begin{definition} \label{DEF_SpecVD}
		We define the spectrum and cardinal invariant associated to Van Douwen families as:
		\begin{align*}
			\spec(\av) &:= \set{\sizeof{\calF}}{\calF \subseteq \bairespace \text{ is Van Douwen}},\\
			\av &:= \min \spec(\av).
		\end{align*}
	\end{definition}
	
	We can also characterize Van Douwen families by their associated ideals:

	\begin{definition} \label{DEF_AssociatedIdeal}
		Let $\calF$ be an eventually different family.
		Then its associated ideal is defined by
		\[
			\calI_0(\calF) := \set{A \in \infsubset{\omega}}{\calF \restr A \text{ is not maximal}} \cup \Fin
		\]
	\end{definition}

	If $\calF$ is maximal this is a proper ideal and $\calF$ is Van Douwen iff $\calI_0(\calF) = \Fin$.
	
	\begin{definition} \label{DEF_Productive}
		Let $\calF$ be a med family and $n \geq 1$.
		Then we say that $\calF$ is $n$-productive iff also $\calF^n$ is med.
		If $\calF$ is $n$-productive for all $n \geq 1$, we say that $\calF$ is finitely productive (or $<\!\!\omega$-productive).
		Further, $2$-productive is called square and $3$-productive is called cubic.
	\end{definition}

	We have the following obvious but useful rewriting of the definitions above:
	
	\begin{remark} \label{REM_ProductiveEquivalence}
		Let $\enumone{\calF}{n}$ be med.
		Then $\enumoperator{\calF}{1}{n}{\times}$ is med iff for all $\enumone{x}{n} \in \bairespace$, there are $f_1 \in \calF_1, \dots, f_n \in \calF_n$ such that $\bigcap_{i = 1}^n E(x_i, f_i)$ is infinite.
	\end{remark}
	
	\begin{proof}
		Follows immediately from $E(x_1 \times \dots \times x_n, f_1 \times \dots \times f_n) = \bigcap_{i = 1}^n E(x_i, f_i)$.
	\end{proof}

	\begin{proposition}\label{PROP_VDisProductive}
		Let $\calV$ be Van Douwen and $\calF$ be any med family.
		Then $\calV \times \calF$ is med.
		Further, if $\calF$ is Van Douwen, then $\calV \times \calF$ is also Van Douwen.
		In particular, every Van Douwen family is finitely productive.
	\end{proposition}

	\begin{proof}
		Obviously, products of ed families are ed, so we only need to verify maximality using Remark~\ref{REM_ProductiveEquivalence}.
		Let $x_1, x_2 \in \bairespace$ be given.
		Since $\calF$ is med there is an $f_2 \in \calF$ with $f_2 =^\infty x_2$.
		Since $\calV$ is Van Douwen there is an $f_1 \in \calV$ with $f_1 =^\infty x_1 \restr E(f_2, x_2)$.
		Thus, $E(f_1, x_1) \cap E(f_2, x_2)$ is infinite, proving that $\calV \times \calF$ is med.
		Further, if $\calF$ is additionally Van Douwen, the same argument also works starting with infinite partial functions $x_1, x_2$ with the same domain, proving that $\calV \times \calF$ is Van Douwen.
	\end{proof}

	In fact, the previous proposition gives another unique characterization of Van Douwen families:

	\begin{proposition}\label{PROP_VDEquivalence}
		Let $\calF$ be med such that $\calF \times \calG$ is med for any other med family $\calG$.
		Then $\calF$ is Van Douwen.
	\end{proposition}

	\begin{proof}
		Assume $\calF$ is such a maximal eventually different family and assume it was not Van Douwen.
		So choose an infinite partial function $p : A \to \omega$ such that $\calF \restr A \cup \simpleset{p}$ is ed.
		W.l.o.g.\ we may assume that $B = \omega \setminus A$ is also infinite.
		Let $\calP := \set{p_\alpha}{\alpha < \frc}$ be a med family on $\functionspace{A}{\omega}$, $\calQ := \set{q_\alpha}{\alpha < \frc}$ be an ed family on $\functionspace{B}{\omega}$ and $q:B \to \omega$ with $q_\alpha(n) \neq q(n)$ for all $n \in B$ and $\alpha < \frc$.
		Finally, define a family of functions $\calG := \set{g_\alpha}{\alpha < \frc}$ on $\bairespace$ by
		\[
			g_\alpha := p_\alpha \cup q_\alpha.
		\]
		Clearly, $\calG$ is ed, but also med.
		For if $h \in \bairespace$, by maximality of $\calP$ there is $\alpha < \frc$ with $p_\alpha =^\infty h \restr A$.
		But then also $g_\alpha =^\infty h$.
		But now, let $x,y \in \bairespace$ extend $p$ and $q$, respectively.
		Then for any $f \in \calF$ and $\alpha < \frc$ we have
		\[
			E(x, f) \cap E(y, g_\alpha) \subseteq (A \cap E(p, f)) \cup (B \cap E(q, g_\alpha)).
		\]
		But $E(p,f)$ and $E(q, g_\alpha)$ are finite and empty, respectively.
		Hence, by Remark~\ref{REM_ProductiveEquivalence} $\calF \times \calG$ is not maximal.
	\end{proof}

	Note that we always have $\calI_0(\calF) \cup \calI_0(G) \subseteq \calI_0(\calF \times \calG)$,
	so the previous proof basically boils down to:
	Assume there is an infinite subset $A \in \calI_0(\calF)$.
	Then we can construct a med family $\calG$ with $\omega \setminus A \in \calI_0(\calG)$.
	Hence, by the above inclusion $\calI_0(\calF \times \calG)$ cannot be a proper ideal, so that $\calF \times \calG$ cannot be maximal.
	Note that this characterization shows that Van Douwen families satisfy an even stronger productivity property:
	
	\begin{definition}\label{DEF_I0Productive}
		Let $\calF$ be a med family and $n \geq 1$.
		Then we say that $\calF$ is $n$-$\calI_0$-productive iff $\calI_0(\calF^n) = \calI_0(\calF)$.
		If $\calF$ is $n$-$\calI_0$-productive for all $n \geq 1$, we say that $\calF$ is finitely-$\calI_0$-productive ($<\!\!\omega$-$\calI_0$-productive).
	\end{definition}

	Clearly, $n$-$\calI_0$-productive implies $n$-productive and in general for any med family $\calF$ we have an associated increasing chain of ideals
	\[
		\calI_0(\calF) \subseteq \calI_0(\calF^2) \subseteq \calI_0(\calF^3) \subseteq \dots
	\]
	But by Remark~\ref{PROP_VDisProductive} for a Van Douwen family $\calF$ its entire associated chain is $\Fin$, so we get

	\begin{corollary}\label{COR_VDisI0Productive}
		Van Douwen families are finitely $\calI_0$-productive.
	\end{corollary}

	\section{Many $\calI_0$-productive families}
	
	Since Van Douwen families always exist and are finitely $\calI_0$-productive, there always are finitely $\calI_0$-productive med families.
	However, note that in the sense of complexity Van Douwenness is a rather strong assumption. In fact, there are closed maximal eventually different families whereas Van Douwen families cannot even be analytic, so there is a huge complexity gap between the two notions.
	So for this section, the goal will be to separate productivity and Van Douwenness by constructing many different productive families which do not come from a Van Douwen family.
	This can be made more precise by the following observation:
	
	\begin{proposition}\label{PROP_ProductiveIfRestrictionIsProductive}
		Let $\calF$ med, $n \geq 1$ and assume there is an $A \in \infsubset{\omega}$ such that $\calF \restr A$ is $n$-productive.
		Then $\calF$ is $n$-productive.
	\end{proposition}

	\begin{proof}
		Let $n < \omega$ and $\enumzero{x}{n-1} \in \bairespace$.
		Since $\calF \restr A$ is $n$-productive, $(\calF \restr A)^n$ is med, so choose $\enumzero{f}{n-1} \in \calF$ such that $\bigcap_{i \in n} E(f_i \restr A, x_i \restr A)$ is infinite.
		But then also $\bigcap_{i \in n} E(f_i, x_i)$ is infinite, so by Remark~\ref{REM_ProductiveEquivalence} $\calF$ is $n$-productive.
	\end{proof}

	In particular, we get that any med $\calF$ which can be restricted to a Van Douwen family is finitely productive.
	Note that we can precisely characterize such families by tallness of $\calI_0$.
	Remember that an ideal $\calI$ is tall iff for all $A \in \infsubset{\omega}$ there is a $B \in \calI$ with $A \cap B$ infinite, or equivalently, $\infsubset{A} \cap \calI$ is non-empty.
	For an ideal $\calI$ and $A \in \infsubset{\omega}$ denote $\calI \restr A := \set{B \cap A}{B \in \calI}$.

	\begin{proposition}{\label{PROP_RestrictionToVD}}
		Let $\calF$ be med.
		Then $\calF$ can be restricted to a Van Douwen family iff $\calI_0(\calF)$ is not a tall ideal.
	\end{proposition}

	\begin{proof}
		It is easy to check that we always have $\calI_0(\calF \restr A) = \calI_0(\calF) \restr A$.
		Hence, $\calI_0(\calF)$ is not tall iff there is an $A \in \infsubset{\omega}$ with $A \cap B$ finite for all $B \in \calI_0(\calF)$ iff there is an $A \in \infsubset{\omega}$ with $\calI_0(\calF \restr A) = \calI_0(\calF) \restr A = \Fin$ iff there is an $A \in \infsubset{\omega}$ such that  $\calF \restr A$ is Van Douwen.
	\end{proof}

	For the rest of this section we will show that we may construct med families with tall $\calI_0$-ideal, hence they cannot be restricted to Van Douwen families.
	In fact, we will show that any proper ideal containing $\Fin$ may be realized as the $\calI_0$-ideal of some finitely productive med family.
	Furthermore, they will even be finitely $\calI_0$-productive.
	This directly generalizes the realization of any proper ideal as the $\calI_0$-ideal of some med family by the author in \cite{Schembecker_2026}.
	
	\begin{theorem}\label{THM_ManyI0Productive}
		Assume {\sf MA}($\sigma$-centered). Let $\calI$ be any proper ideal containing $\Fin$.
		Then there is a finitely $\calI_0$-productive med family $\calF$ with $\calI_0(\calF) = \calI$.
	\end{theorem}
	
	Note that this implies that $\calI_0(\calF^n) = \calI$ for all $n \geq 1$, so in other words its associated increasing chain of ideals is constantly $\calI$.
	Compared to the proof in \cite{Schembecker_2026} we take a more forcing-centered approach here.
	Fix a proper ideal $\calI$ containing $\Fin$, we will iteratively construct an increasing sequence of ed families $\seq{\calF_\alpha}{\alpha < \frc}$ with $\sizeof{\calF_\alpha} < \frc$ and a sequence of partial functions $\seq{p_\alpha:A_\alpha \to \omega}{\alpha < \frc}$, where $\seq{A_\alpha}{\alpha < \frc}$ enumerates $\calI \setminus \Fin$, such that for all $\alpha, \beta < \frc$ we have
	\[
		\calF_\alpha \cup \simpleset{p_\beta} \text{ is an ed family}.
	\]
	Then $\calF = \bigcup_{\alpha < \frc} \calF_\alpha$ is ed and the $p_\alpha$ witness that $\calI \subseteq \calI_0(\calF)$.
	By choosing $f_\alpha$ generic enough we will also have that $\calI_0(\calF^n) \subseteq \calI$ for all $n \geq 1$, which shows that $\calF$ is finitely $\calI_0$-productive.
	First, we discuss the forcing adding $p_\alpha$.
	
	\begin{definition}
		Let $A \in \infsubset{\omega}$ and $\calF \subseteq \bairespace$, then the forcing $\bbE(\calF \restr A)$ is given by conditions $(s, E)$, where $s:A \partialto \omega$ is a finite partial function and $E \in \finsubset{\calF}$.
		For $(s, E), (t, F) \in \bbE(\calF \restr A)$ we define $(t, F) \extends (s, E)$ if $F \supseteq E$, $t \supseteq s$ and for all $k \in \dom(t \setminus s)$ and $f \in E$ we have $t(k) \neq f(k)$.
	\end{definition}
	
	\noindent It is well-known that $\bbE(\calF \restr A)$ is $\sigma$-centered, so given $\calF_\alpha$ and $A_\alpha$ by {\sf MA}($\sigma$-centered) we may choose $G$ generic for the dense sets
	\[
		D_k := \set{(s,E)}{k \in \dom(s)},
	\]
	where $k \in A$ and
	\[
		D_f := \set{(s,E)}{f \in E},
	\]	
	where $f \in \calF$.
	Then $p_\alpha := \bigcup\set{s}{(s, E) \in G}$ is a function $p_\alpha:A \to \omega$ eventually different from all $f \in \calF$.
	Next, let $\seq{\vec{h}_\alpha}{\alpha < \frc}$ enumerate all tuples $\vec{h} = (\enumzero{h}{n-1})$ of partial functions $h_i:B \to \omega$ for some common domain $B \notin \calI$.
	In order to ensure $B \notin \calI_0(\calF^n)$, and thus eventually $\calI_0(\calF^n) \subseteq \calI$, we will need to have functions $\enumzero{f}{n-1} \in \calF_{\alpha + 1}$ such that
	\[
		\bigcap_{i\in n} E(h_i, f_i)
	\]
	is infinite (see Definition~\ref{DEF_AssociatedIdeal} and Remark~\ref{REM_ProductiveEquivalence}).
	So let $\vec{h}_\alpha = (\enumzero{h}{n-1})$, then we will define a decreasing sequence of infinite subsets of $B$ not in $\calI$
	\[
		B =: X_0 \supseteq X_1 \supseteq \dots \supseteq X_n.
	\]
	So assume $X_i \notin \calI$ has been defined.
	If there is an $f_i \in \calF_{\alpha}$ such that $X_i \cap E(h_i, f_i) \notin \calI$, then we call $i$ an old coordinate witnessed by $f_i$ and set $X_{i+1} := X_i \cap E(h_i, f_i)$.
	Otherwise, we call $i$ a new coordinate and set $X_{i+1} := X_i$.
	At the end, we have $X_n \notin \calI$ and we let $S \subseteq n$ be the set of all new coordinates.
	By construction for all $i \in S$,  $f \in \calF_\alpha$ we have
	\[
		X_n \cap E(h_i, f) \in \calI.
	\]
	Now, define an equivalence relation $\equiv$ on $X_n$ by 
	\[
		k \equiv l \quad \text{iff} \quad \forall i, j \in S \text{ we have } (h_i(k) = h_j(k) \iff h_i(l) = h_j(l)).
	\]
	Since $S$ is finite, there are only finitely many equivalence classes, so because $X_n \notin \calI$ we may choose an equivalence class $X \subseteq X_n$ with $X \notin \calI$.
	Now, we define a second equivalence relation $\sim$ on $S$ by
	\[
		i \sim j \quad \text{iff} \quad \forall k \in X \text{ we have } h_i(k) = h_j(k).
	\]
	Finally, let $S/\!\!\sim\ = m = \set{a}{a \in m}$, for every $a \in m$ choose a representative $i_a \in S$ and denote with $[i]$ the equivalence class of $i$.
	Then, by construction for all $a \neq b \in m$ and $k \in X$ we have
	\[
		h_{i_a}(k) \neq h_{i_b}(k).
	\]
	If $S = \emptyset$, then $\calF_\alpha$ already diagonalizes against $\vec{h}_\alpha$, so assume $S \neq \emptyset$.
	Then we will add $m$ new functions $g_a:\omega\to\omega$ to $\calF_\alpha$, one for every equivalence class $a \in m$.
	
	\begin{definition}
		Let $J \neq \emptyset$ be finite and $\calF$ be a set of infinite partial functions from $\omega$ to $\omega$.
		Then the forcing $\bbE_J(\calF)$ is given by conditions $(s, E)$, where $s:J\times\omega \partialto \omega$ is a finite partial function, $E \in \finsubset{\calF}$, $\dom(s) = J \times D$ for some $D \in \finsubset{\omega}$ and for all $(i,k) \neq (j,k) \in \dom(s)$ we have $s(i,k) \neq s(j,k)$.
		For $(s,E), (t,F) \in \bbE_J(\calF)$ we define $(t,F) \extends (s,E)$ if $F \supseteq E$, $t \supseteq s$ and for all $f \in E$ and $(i, k) \in \dom(t \setminus s)$ with $k \in \dom(f)$ we have $t(i,k) \neq f(k)$.
	\end{definition}
	
	Again $\bbE_J(\calF)$ is $\sigma$-centered, so given $\calF := \calF_{\alpha} \cup \set{p_\beta}{\beta \leq \alpha}$ and $J := m$, by {\sf MA}($\sigma$-centered) we may choose $G$ generic for the dense sets
	\[
		D_k := \set{(s,E)}{(i,k) \in \dom(s) \text{ for all } i \in J},
	\]
	where $k \in \omega$ and
	\[
		D_f := \set{(s,E)}{f \in E},
	\]
	where $f \in \calF$ and
	\[
		E_K := \set{(s,E)}{\exists k > K \text{ with } k \in X \text{ such that } s(a,k) = h_{i_a}(k) \text{ for all } a \in m},
	\]
	where $K \in \omega$.
	The only non-trivial density argument is for $E_K$, so let $K \in \omega$ and $(s,E) \in \bbE_J(\calF)$ with $\dom(s) = J \times D$.
	Let $E = E_0 \cup E_1$, where $E_0 \in \finsubset{\calF_\alpha}$ and $E_1 = \set{p_{\beta}:A_\beta\to\omega}{\beta \in B}$ for some $B \in \finsubset{\alpha + 1}$.
	Remember that by construction of $X_n$ for all $f \in E_0$ the set
	\[
		X \cap E(h_{i_a}, f) \in \calI.
	\]
	Hence, we have that
	\[
		\left(\bigcup_{f \in E_0, a \in m} (X \cap E(h_{i_a}, f)) \cup \bigcup_{\beta \in B} A_\beta\right) \in \calI.
	\]
	But $X \notin \calI$, so we may pick $k \in X$ with $k > K$, $k \notin D$ and
	\begin{enumerate}[$\circ$]
		\item $h_{i_a}(k) \neq f(k)$ for all $a \in m$ and $f \in E_0$,
		\item $k \notin \dom(p_\beta)$ for all $\beta \in B$.
	\end{enumerate}
	But then we may extend $s$ by $s(a,k) := h_{i_a}(k)$ for all $a \in m$ as required.
	Note that $h_{i_a}(k) \neq h_{i_b}(k)$ for $a \neq b$, since they represent different equivalence classes.
	Now, given a generic $G$ for these $<\!\!\frc$-many dense sets, for $a \in m$ we let
	\[
		g_a := \bigcup\set{s_a}{(s, E) \in G},
	\]
	where $s_a$ is the $a$-th component of $s$ and we set $\calF_{\alpha + 1} := \calF_{\alpha} \cup \set{g_a}{a \in m}$.
	Then $\calF_{\alpha + 1}$ is still eventually different, for all $\beta \leq \alpha$ we have $\calF_{\alpha + 1} \cup \simpleset{p_\beta}$ is still eventually different and
	\[
		\bigcap_{i \in S} E(h_i, g_{[i]}) \cap \bigcap_{i \in n \setminus S} E(h_i, f_i)
	\]
	is infinite. Now, Theorem~\ref{THM_ManyI0Productive} follows with a standard bookkeeping argument.\hfill\qedsymbol\\
	
	Finally, note that whereas the starting ideal $\calI_0(\calF)$ of the associated chain of ideals may be any proper ideal containing $\Fin$, the complete chain cannot be any completely arbitrary chain: At least if $\calI_0(\calF) = \Fin$, then by the discussion above $\calF$ is Van Douwen and the entire associated chain of ideals is $\Fin$.
	
	\begin{question}
		What structure and restrictions does the associated chain of ideals of a maximal eventually different families have?
	\end{question}
	
	\section{Borel productive families}
	
	Remember that Van Douwen families give canonical examples of finitely productive families.
	However, Raghavan proved that Van Douwen families can never be analytic, whereas Horowitz and Shelah constructed a Borel maximal eventually different family $\calE$ in \cite{HorowitzShelah_2024}.
	Hence, we may naturally ask what the possible complexity of productive families is.
	We will first show that $\calE$ is not square.
	However, we will also show that their construction may be generalized to obtain a Borel $n$-productive and not $(n+1)$-productive family $\calE_n$ for every $n \geq 1$.
	Further, we will also construct a Borel finitely productive family $\calE_{<\omega}$.
	Hence, we may separate all levels of productivity by Borel witnesses.
	We begin with a short reminder of the construction by Horowitz and Shelah.
	We highly recommend the exposition by Schrittesser in \cite{Schrittesser_2016} and will also partially follow their notation for the new proofs of this chapter.
	
	First of all, instead of constructing a family of functions from $\omega$ to $\omega$ in the context of the Borel maximal eventually different family it will be more natural to instead consider functions from $\omega$ to $\finseq{\omega}$.
	In order to avoid double exponents we write $\Seq := \finseq{\omega}$ for the set of all finite sequences of natural numbers,
	so the med families we construct will be subsets of $\infseq{\Seq} = \infseq{(\finseq{\omega})}$.
	Fix a computable bijection $\Phi:\infseq{\Seq} \to \bairespace$ with computable inverse.
	Now, there is an obvious naive eventually different family indexed by $\bairespace$:
	\[
		\set{\text{e}(f)}{f \in \bairespace},
	\]
	where $\text{e}(f)$ is the restriction function defined by $\text{e}(f)(n) := f \restr n$.
	Clearly, this family is eventually different but will not be maximal.
	Any given $h: \omega \to \Seq$ need not be infinitely often equal to some $\text{e}(f)$.
	In this case we want to modify the function indexed by $\Phi(h)$, that is $\text{e}(\Phi(h))$, to be infinitely often equal to $h$ by overwriting some part of it by a copy of $h$, and thus witnessing maximality.
	Hence, for a given $h \in \infseq{\Seq}$ we define an infinite subset
	\[
		B(h) := \set{2 \langle h \restr n \rangle + 1}{n \in \omega},
	\]
	where $\langle \cdot \rangle$ is a computable coding of finite sequences of members in $\Seq$ by natural numbers.
	Note that $\set{B(h)}{h \in \infseq{\Seq}}$ is an almost disjoint family.
	Now, we would like to replace $\text{e}(\Phi(h))$ on $B(h)$ by a copy of $h$.
	Doing this for all $h \in \infseq{\Seq}$ would yield a maximal family, but unfortunately this is not eventually different any more. 
	So the key idea is to do the copying only if it is safe to do so on a smaller infinite subset $C(h) \subseteq B(h)$.
	Then $\set{\dot{\text{e}}(h, C(h))}{h \in \infseq{\Seq}}$ will be maximal eventually different, where
	\[
		\dot{\text{e}}(h, C)(n) =
		\begin{cases}
			h(n) &\text{if } n \in C,\\
			\text{e}(\Phi(h))(n) & \text{otherwise},
		\end{cases}
	\]
	and $C(h)$ is either an infinite subset of $B(h)$ or $C(h) = \emptyset$ if we do not copy.
	We never modify anything on even coordinates, so the even coordinates can still recover $h$, which is important for definability.
	Copying $h$ on all of $B(h)$ would not in general be safe, because $h$ may have infinite agreement with some code $\text{e}(f)$ on $B(h)$.
	Instead, we look for an infinite subset $C(h) \subseteq B(h)$ on which $h$ has only finite, in fact at most one, agreement with every code $\text{e}(f)$.
	Such a set is safe for copying.
	More precisely, we can arithmetically detect the safety of copying using the following relation $\prec_h$:
	
	\begin{definition}
		Let $h \in \infseq{\Seq}$.
		Define a relation $\prec_h$ on $B(h)$ by $m \prec_h m'$ iff
		\[
			m < m', \quad \length(h(m)) = m, \quad \length(h(m')) = m' \quad \text{ and } \quad h(m) \subsetneq h(m').
		\]
	\end{definition}

	\noindent We define the set of all $\prec_h$-maximal (or terminal) points to be
	\[
		T(h) := \set{m \in B(h)}{\forall m' \in B(h) \setminus m \text{ we have } m \not\prec_h m'}.
	\]
	If $T(h)$ is infinite, we set $C(h) := T(h)$ and call $h$ copyable.
	Then, by definition of $\prec_h$, on $T(h)$ the function $h$ cannot agree with any $\text{e}(f)$ on more than one point, so we are safe to copy $h$ onto $\text{e}(\Phi(h))$ via $\dot{\text{e}}(h, C(h))$.
	
	Otherwise, there is an infinite strictly increasing chain
	\[
		m_0 \prec_h m_1 \prec_h m_2 \prec_h \dots,
	\]
	of elements in $B(h)$:
	Indeed, because $T(h)$ is finite, we let $m_0$ denote the minimal element of $B(h)$ such that $h(m_0)$ is longer than all $h(m)$, where $m \in T(h)$.
	Clearly, such an $m_0$ exists, for if otherwise $T(h)$ would be infinite.
	Then, we can define $m_{i+1}$ to be the minimal element of $B(h)$ with $m_{i} \prec_h m_{i+1}$.
	This is possible since $m_i \notin T(h)$.
	
	But then, we also obtain an infinite strictly increasing chain
	\[
		h(m_0) \subsetneq h(m_1) \subsetneq h(m_2) \subsetneq \dots,
	\]
	with union $f \in \bairespace$.
	Now, we have
	\[
		h(m_i) = f \restr m_i = \text{e}(f)(m_i),
	\]
	so $h$ is infinitely often equal to $\text{e}(f)$ on $B(h)$.
	We say that $h$ is coherent with $f$ and set $C(h) := \emptyset$.
	Thus, the construction uses the arithmetical sufficient condition $T(h) \in \infsubset{\omega}$ for safe copying.
	Note that copyability and coherence with some $f$ are not mutually exclusive; sometimes we copy even though $h$ is already caught by a code.
	
	Finally, one can now show that the family $\calE := \set{\dot{\text{e}}(h, C(h))}{h \in \infseq{\Seq}}$ is med, and because $h$ can be recovered from $\dot{\text{e}}(h, C(h))$, copyability is arithmetical and $C(h)$ can be found arithmetically in $h$, in fact the final family $\calE$ will be arithmetical (see \cite{Schrittesser_2016} for the full details).
	With all definitions in place, we can finally consider productiveness of $\calE$.
	In fact, our choice of codomain as $\Seq$ makes the proof very natural:
	
	\begin{lemma} \label{LEM_HSNotSquare}
		The Horowitz-Shelah med family $\calE$ is not square.
	\end{lemma}

	\begin{proof}
		Define for each $n \in \omega$ the pairwise disjoint sets $D_n := {{^n}\omega} \subseteq \Seq$.
		Then, $D_n$ is precisely the range of the map $f \mapsto \text{e}(f)(n)$.
		Choose two functions $x_0, x_1 \in \infseq{\Seq}$ such that for all $n \in \omega$ we have
		\[
			x_0(n) \neq x_1(n) \quad \text{and} \quad x_0(n), x_1(n) \notin D_n.
		\]
		We claim that $x_0, x_1$ witness that $\calE$ is not square, so let $h_0, h_1 \in \infseq{\Seq}$ and assume that the intersection $E(x_0, \dot{\text{e}}(h_0, C(h_0))) \cap E(x_1, \dot{\text{e}}(h_1, C(h_1)))$ is infinite.
		First, note that for $i \in 2$ we have 
		\[
			E(x_i, \dot{\text{e}}(h_i, C(h_i))) \subseteq C(h_i).
		\]
		Indeed, for $n \in \omega \setminus C(h_i)$ we have
		\[
			\dot{\text{e}}(h_i, C(h_i))(n) = e(\Phi(h_i))(n) \in D_n,
		\]
		which cannot equal $x_i(n) \notin D_n$.
		But the $C(h_i)$'s are almost disjoint, so $h_0 = h_1$.
		But this would imply that $x_0(n) = x_1(n)$ for infinitely many $n$, contradicting the choice of $x_0$ and $x_1$.
	\end{proof}

	Fundamentally, squareness breaks down, since the construction of $\calE$ only diagonalizes against any possible violation of maximality given by one element of $\infseq{\Seq}$.
	In other words, if we set $\calR := \infseq{\Seq}$, then $\calR$ is the set of all requirements we need to diagonalize against to obtain a maximal family indexed by the Polish space $\calR$.
	Now, in order to obtain a square med family $\calE_2$ we need to diagonalize against requirements given by pairs $(h_0, h_1) \in \infseq{\Seq} \times \infseq{\Seq}$.
	So, naively one might try to set $\calR_2 = \infseq{\Seq} \times \infseq{\Seq}$ and construct a square family indexed by the Polish space $\calR_2$.
	However, in order for $(h_0, h_1) \in \infseq{\Seq} \times \infseq{\Seq}$ to not be a violation the final family needs to have two members $f_0, f_1 \in \calE_2$ such that
	\[
		E(h_0, f_0) \cap E(h_1, f_1)
	\]
	is infinite, so we really need to add two members per requirement to diagonalize against every possible violation.
	Thus, the correct indexing set for our family $\calE_2$ is $\calR_2 := \infseq{\Seq} \times \infseq{\Seq} \times 2$.
	However, crucially the almost disjoint supports for copying should be indexed by the requirement $(h_0, h_1) \in \infseq{\Seq} \times \infseq{\Seq}$, so the functions associated to $(h_0, h_1, 0)$ and $(h_0, h_1, 1)$ share the same support, where copying may happen.
	We now describe and prove the full construction of the Borel square med family $\calE_2$, in order to present all new ideas in the simplest setting. We then explain how to further generalize the construction to obtain Borel $n$-productive and even Borel finitely productive med families.
	The ideas are similar to the original proof: Maximality is again obtained by copying functions on suitably chosen almost disjoint supports, however the considerations when it is safe to do such copying and on which set are considerably more involved, since for a given requirement $(h_0, h_1)$ there are now four different cases, according to whether $h_0, h_1$ are copyable or not.
	
	First, we fix a computable bijection $\Phi_2:\calR_2 \to \bairespace$ with computable inverse $\Psi_2$.
	The initial unmodified eventually different family is given by
	\[
		\set{\text{e}(\Phi_2(h_0,h_1,i))}{(h_0,h_1,i) \in \calR_2},
	\]
	where $\text{e}(f)(n) := f \restr n$.
	As before, we will replace some parts of $\text{e}(\Phi_2(h_0,h_1,i))$ with copies of parts of $h_0$ and $h_1$.
	Denote with $h_0 \ast h_1$ the standard coding of pairs given by
	\begin{align*}
		(h_0 \ast h_1)(2n) &:= h_0(n),\\
		(h_0 \ast h_1)(2n + 1) &:= h_1(n).
	\end{align*}
	Again, we define an almost disjoint family of copy supports, but now indexed by $\infseq{\Seq} \times \infseq{\Seq}$:
	\[
		B(h_0,h_1) := B(h_0 \ast h_1) = \set{2\langle (h_0 \ast h_1) \restr n \rangle + 1}{n \in \omega}.
	\]
	We will again arithmetically find some $C(h_0,h_1) \subseteq B(h_0,h_1)$, so that for $i \in 2$ we have that $\dot{\text{e}}(h_0,h_1,i)$ will mostly be $\text{e}(\Phi_2(h_0,h_1,i))$, except possibly some copying of $h_0$ or $h_1$ on $C(h_0,h_1)$.
	Again, we never modify even coordinates, so we will have that $\Phi_2(h_0,h_1,i)$ and thus $h_0,h_1$ and $i$ can be recovered from $\dot{\text{e}}(h_0,h_1,i)$, which is important for definability.
	As before, for $h \in \infseq{\Seq}$ consider the relation $\prec_h$ on $B(h_0,h_1)$ defined by $m \prec_h m'$ iff
	\[
		m < m', \quad \length(h(m)) = m, \quad \length(h(m')) = m' \quad \text{ and } \quad h(m) \subsetneq h(m').
	\]
	Furthermore, for an infinite subset $X \in \infsubset{B(h_0,h_1)}$ define the set of $\prec_h$-maximal points restricted to $X$ by
	\[
		T_X(h) := \set{m \in X}{\forall m' \in X \setminus m \text{ we have } m \not\prec_h m'}.
	\]
	As before, either $T_X(h)$ is infinite, or if $T_X(h)$ is finite there is a strictly increasing chain
	\[
		m_0 \prec_h m_1 \prec_h m_2 \prec_h \dots
	\]
	of elements of $X$.
	Note that such a chain can be chosen arithmetically if we always choose minimal $m_i \in X$.
	As described before, in this case we get a strictly increasing chain
	\[
		h(m_0) \subsetneq h(m_1) \subsetneq h(m_2) \subsetneq \dots
	\]
	with union $f \in \bairespace$.
	Then, again
	\[
		h(m_i) = f \restr m_i = \text{e}(f)(m_i),
	\]
	so $h$ is infinitely often equal to some $\text{e}(f)$ on $X$.
	We need this restricted version, since we will thin out sequentially to smaller and smaller infinite subsets of $B(h_0, h_1)$.
		
	Now, fix some requirement $(h_0,h_1) \in \infseq{\Seq} \times \infseq{\Seq}$ and start with $X_0 := B(h_0,h_1)$.
	If we have that $T_{X_0}(h_0)$ is infinite, set $X_1 := T_{X_0}(h_0)$ and call $h_0$ copyable for this requirement.
	Otherwise, let $X_1$ be the canonical strictly increasing chain in $X_0$, let $f_{h_0}$ be the corresponding union of $h_0(m_i)$ for $m_i \in X_1$ as described above and call $h_0$ coherent with $f_{h_0}$ for this requirement.
	
	Now, repeat with $X_1$:
	If we have that $T_{X_1}(h_1)$ is infinite, set $X_2 := T_{X_1}(h_1)$ and call $h_1$ copyable for this requirement.
	Otherwise, let $X_2$ be the canonical strictly increasing chain in $X_1$, let $f_{h_1}$ be the corresponding union of $h_1(m_i)$ for $m_i \in X_2$ as described above and call $h_1$ coherent with $f_{h_1}$ for this requirement.
	Note that the notions of copyability and coherence are relative to the current support.
	Thus, they are not symmetric in the two coordinates: $h_0$ is tested on $X_0$, whereas $h_1$ is tested only after thinning to $X_1$.
	
	Finally, if both $h_0$ and $h_1$ are copyable for this requirement we must ensure that the two copied functions remain eventually different.
	If $E(h_0,h_1) \cap X_2$ is infinite, then we set
	\[
		C(h_0,h_1) := E(h_0,h_1) \cap X_2.
	\]
	In this case a single copied function suffices to witness both coordinates.
	Otherwise, $X_2 \setminus E(h_0,h_1)$ is infinite and we set
	\[
		C(h_0,h_1) := X_2 \setminus E(h_0,h_1).
	\]
	Then $h_0$ and $h_1$ disagree on every point of the common support $C(h_0,h_1)$, so we can safely copy both of them on the common support $C(h_0, h_1)$.
	In all other cases we simply set
	\[
		C(h_0, h_1) := X_2.
	\]
	
	Finally, we describe how to construct $\dot{\text{e}}(h_0,h_1,0)$ and $\dot{\text{e}}(h_0,h_1,1)$:
	If $h_0$ is coherent with $f_{h_0}$ and $h_1$ is coherent with $f_{h_1}$ we do not copy anything and for $i \in 2$ just set
	\[
		\dot{\text{e}}(h_0,h_1,i) := \text{e}(\Phi_2(h_0,h_1,i)).
	\]
	If $h_0$ is copyable and $h_1$ is coherent with $f_{h_1}$, then we need to copy $h_0$, but we need to be careful to not overwrite $f_{h_1}$ accidentally in case $f_{h_1}$ happens to be either $\Phi_2(h_0,h_1,0)$ or $\Phi_2(h_0,h_1,1)$.
	However, as we are adding two functions we let $i \in 2$ be minimal such that $\Phi_2(h_0,h_1,i) \neq f_{h_1}$.
	Now, copy $h_0$ onto the $i$-th function and leave the other function unmodified:
	\begin{align*}
		\dot{\text{e}}(h_0,h_1,i)(n) &:=
		\begin{cases}
			h_0(n) & \text{if } n \in C(h_0,h_1),\\
			\text{e}(\Phi_2(h_0,h_1,i))(n) & \text{otherwise}.
		\end{cases}\\
		\dot{\text{e}}(h_0,h_1,1-i) &:= \text{e}(\Phi_2(h_0,h_1,1-i))
	\end{align*}
	If $h_0$ is coherent with $f_{h_0}$ and $h_1$ is copyable this is completely analogous to the previous case, so assume that both $h_0$ and $h_1$ are copyable.
	Then, we have two subcases: If $C(h_0,h_1) \subseteq E(h_0,h_1)$, then we just copy $h_0 \restr C(h_0,h_1) = h_1 \restr C(h_0,h_1)$ onto the first function and leave the other function unmodified:
	\begin{align*}
		\dot{\text{e}}(h_0,h_1,0)(n) &:=
		\begin{cases}
			h_0(n) (= h_1(n)) & \text{if } n \in C(h_0,h_1),\\
			\text{e}(\Phi_2(h_0,h_1,0))(n) & \text{otherwise}.
		\end{cases}\\
		\dot{\text{e}}(h_0,h_1,1) &:= \text{e}(\Phi_2(h_0,h_1,1))
	\end{align*}
	Otherwise, $h_0(n) \neq h_1(n)$ for all $n \in C(h_0,h_1)$ and we copy $h_0$ onto the first function and $h_1$ onto the second function:
	\begin{align*}
		\dot{\text{e}}(h_0,h_1,0)(n) &:=
		\begin{cases}
			h_0(n) & \text{if } n \in C(h_0,h_1),\\
			\text{e}(\Phi_2(h_0,h_1,0))(n) & \text{otherwise}.
		\end{cases}\\
		\dot{\text{e}}(h_0,h_1,1)(n) &:=
		\begin{cases}
			h_1(n) & \text{if } n \in C(h_0,h_1),\\
			\text{e}(\Phi_2(h_0,h_1,1))(n) & \text{otherwise}.
		\end{cases}
	\end{align*}
	This finishes the construction of $\calE_2 := \set{\dot{\text{e}}(h_0,h_1,i)}{(h_0,h_1,i) \in \calR_2}$.
	It remains to show that $\calE_2$ is indeed a Borel square med family.
	
	\begin{theorem}
		$\calE_2$ is a Borel square med family.
	\end{theorem}
	
	\begin{proof}
		For $(h_0,h_1,i) \in \calR_2$ we let $C(h_0,h_1,i) := C(h_0,h_1)$ if $\dot{\text{e}}(h_0,h_1,i)$ was modified and set $C(h_0,h_1,i) := \emptyset$ otherwise.
		So $\dot{\text{e}}(h_0,h_1,i)$ equals $\text{e}(\Phi_2(h_0,h_1,i))$ outside $C(h_0,h_1,i)$ and is a copy of $h_0$ or $h_1$ on $C(h_0,h_1,i)$.
		
		First, we argue that $\calE_2$ is eventually different, so let $(h_0,h_1,i) \neq (g_0,g_1,j) \in \calR_2$.
		Outside $C(h_0,h_1,i) \cup C(g_0,g_1,j)$ we just have $\text{e}(\Phi_2(h_0,h_1,i))$ versus $\text{e}(\Phi_2(g_0,g_1,j))$, which are eventually different as $(h_0,h_1,i) \neq (g_0,g_1,j)$.
		
		Next, on $C(h_0,h_1,i) \setminus C(g_0,g_1,j)$ we have a copy of $h_0$ or $h_1$ against $\text{e}(\Phi_2(g_0,g_1,j))$.
		Then, by definition of copyable we have that both sides may only agree on at most one value.
		The case $C(g_0,g_1,j) \setminus C(h_0,h_1,i)$ is symmetric.
		
		It remains to consider $C(h_0,h_1,i) \cap C(g_0,g_1,j)$.
		This set is finite if $(h_0,h_1) \neq (g_0,g_1)$, so we may further assume that $(h_0,h_1) = (g_0,g_1)$, so $i \neq j$.
		But the only case, where we modified both functions attached to the same requirement $(h_0,h_1)$ is when we additionally arranged that $h_0(n) \neq h_1(n)$ for all $n \in C(h_0,h_1)$ and copied $h_0$ onto the first function and $h_1$ onto the second function.
		Hence, $\calE_2$ is eventually different.
		
		For squareness, let $h_0,h_1 \in \infseq{\Seq}$.
		If $h_0$ is coherent with $f_{h_0}$ and if $h_1$ is coherent with $f_{h_1}$, then by thinning $X_0$ to $X_1$ we ensured that
		\[
			\text{e}(f_{h_0}) \restr C(h_0,h_1) = h_0 \restr C(h_0,h_1).
		\]
		Similarly, by thinning $X_1$ to $X_2$ we ensured that
		\[
			\text{e}(f_{h_1}) \restr C(h_0,h_1) = h_1 \restr C(h_0,h_1).
		\]
		Now, let $\Psi_2(f_{h_0}) = (g_0, g_1, i)$.
		If $(h_0, h_1) = (g_0, g_1)$, then we did not modify $\text{e}(f_{h_0})$ so that
		\[
			E(h_0, \dot{\text{e}}(g_0,g_1,i)) \supseteq C(h_0,h_1).
		\]
		Otherwise, $(h_0, h_1) \neq (g_0, g_1)$.
		But then $f_{h_0}$ could only have been modified on $C(g_0, g_1)$ which is almost disjoint from $C(h_0, h_1)$.
		Hence, also in this case
		\[
			E(h_0, \dot{\text{e}}(g_0,g_1,i)) \supseteq^* C(h_0,h_1).
		\]
		By the same argument for $\Psi_2(f_{h_1}) = (g_0',g_1',i')$ we get
		\[
			E(h_1, \dot{\text{e}}(g_0',g_1',i')) \supseteq^* C(h_0,h_1).
		\]
		Together, this shows that 
		\[
			E(h_0, \dot{\text{e}}(g_0,g_1,i)) \cap E(h_1, \dot{\text{e}}(g_0',g_1',i')) \supseteq^* C(h_0,h_1)
		\]
		is infinite.
		Next, assume that $h_0$ is copyable and $h_1$ is coherent with $f_{h_1}$.
		We still have 
		\[
			\text{e}(f_{h_1}) \restr C(h_0,h_1) = h_1 \restr C(h_0,h_1).
		\]
		Further, we modified $\dot{\text{e}}(h_0,h_1, i)$ for some $i \in 2$ with $\Phi_2(h_0,h_1,i) \neq f_{h_1}$, so that
		\[
			\dot{\text{e}}(h_0,h_1,i) \restr C(h_0,h_1) = h_0 \restr C(h_0,h_1).
		\]
		Now, for $\Psi_2(f_{h_1}) = (g_0,g_1,j)$ we either have $(g_0,g_1) = (h_0, h_1)$ and $i \neq j$, or $(h_0, h_1) \neq (g_0, g_1)$, so by a similar argument as before
		\[
			E(h_1, \dot{\text{e}}(g_0,g_1,j)) \supseteq^* C(h_0,h_1).
		\]
		Together, this shows that
		\[
			E(h_0, \dot{\text{e}}(h_0,h_1,i)) \cap E(h_1, \dot{\text{e}}(g_0,g_1,j)) \supseteq^* C(h_0,h_1)
		\]
		is infinite.
		The case where $h_0$ is coherent and $h_1$ is copyable is symmetric.
		Finally, assume that both $h_0$ and $h_1$ are copyable.
		Then we have two subcases, so let us first assume that $C(h_0,h_1) \subseteq E(h_0,h_1)$.
		Then we copied $h_0 \restr C(h_0,h_1) = h_1 \restr C(h_0,h_1)$ onto $\text{e}(\Phi_2(h_0,h_1,0))$, so
		\[
			E(h_0, \dot{\text{e}}(h_0,h_1,0)) \cap E(h_1, \dot{\text{e}}(h_0,h_1,0)) \supseteq C(h_0, h_1)
		\] 
		is infinite.
		Otherwise, $h_0$ and $h_1$ are distinct on $C(h_0, h_1)$ and we copied $h_0 \restr C(h_0,h_1)$ onto $\text{e}(\Phi_2(h_0,h_1,0))$ and $h_1 \restr C(h_0,h_1)$ onto $\text{e}(\Phi_2(h_0,h_1,1))$, so
		\[
			E(h_0, \dot{\text{e}}(h_0,h_1,0)) \cap E(h_1, \dot{\text{e}}(h_0,h_1,1)) \supseteq C(h_0, h_1)
		\]
		is infinite.
		It remains to argue that $\calE_2$ is Borel, we even argue that it is arithmetic.
		The exact complexity does not really matter as we construct a closed witness from $\calE_2$ in the next section.
		Since we never modify even coordinates, for $h_0, h_1 \in \infseq{\Seq}$, $i \in 2$ and $n \in \omega$ we have that
		\[
			\dot{\text{e}}(h_0,h_1,i)(2n) = \text{e}(\Phi_2(h_0,h_1,i))(2n) = \Phi_2(h_0,h_1,i) \restr 2n.
		\]
		Therefore, from the even coordinates of $\dot{e}(h_0,h_1,i)$ one can arithmetically recover $\Phi_2(h_0,h_1,i)$ and thus $h_0, h_1$ and $i$.
		Furthermore, we have the following:
		\begin{enumerate}[$\circ$]
			\item $B(h_0,h_1)$ is arithmetical in $h_0$ and $h_1$,
			\item The relation $m \prec_h m'$ is arithmetical in $h, m, m'$,
			\item The set $T_X(h)$ is arithmetical in $h$ and $X$,
			\item The copy or coherence case is decided by the property \textquoteleft$T_X(h)$ is infinite\textquoteright, which is arithmetical in $T_X(h)$,
			\item In the copyable case, the next support is $T_X(h)$, which is arithmetical in $h$ and $X$,
			\item In the non-copyable case, $T_X(h)$ is finite, so the canonical chain, and thus the next support, is arithmetical in $h$ and $X$ by least choices,
			\item The branch $f_h$ given by the union along the canonical chain is arithmetical in $h$ and the current support $X$,
			\item Therefore, the thinning $B(h_0, h_1) = X_0 \mapsto X_1$, the decision whether $h_0$ is copyable or coherent and the possible canonical branch $f_{h_0}$ are all arithmetical in $h_0$ and $h_1$,
			\item By the same argument, the thinning $X_1 \mapsto X_2$, the decision whether $h_1$ is copyable or coherent and the possible canonical branch $f_{h_1}$ are all arithmetical in $h_0$ and $h_1$,
			\item In the double copyable case, the decision property \textquoteleft$E(h_0,h_1) \cap X_2$ is infinite\textquoteright\ is arithmetical in $h_0, h_1$ and $X_2$ and thus the final choice of $C(h_0,h_1)$ is arithmetical in $h_0$ and $h_1$,
			\item The decision which of the two members attached to $(h_0, h_1)$ are modified is arithmetical in $h_0, h_1$.
			Note that in the mixed cases this only requires checking for $i \in 2$ if
			\[
				\Phi_2(h_0,h_1,i) \neq f_{h_j},
			\]
			where $f_{h_j}$ is the coherent branch.
			Picking the smallest such $i \in 2$ is arithmetical in $h_0$ and $h_1$,
			\item Hence, the definition of $\dot{\text{e}}(h_0,h_1,i)$ is arithmetical in $h_0, h_1$ and $i$.
		\end{enumerate}
		But this shows that membership $g \in \calE_2$ is arithmetic:
		Given $g$, check if $\length(g(2n)) = 2n$, and if so recover the unique parameters $(h_0,h_1,i)$ from the even coordinates and then check arithmetically, if $g = \dot{\text{e}}(h_0,h_1,i)$.
	\end{proof}
	
	\noindent By a similar argument as before we also have that $\calE_2$ is not cubic.
	
	\begin{lemma}\label{LEM_E2NotCubic}
		$\calE_2$ is not cubic.
	\end{lemma}

	\begin{proof}
		As before choose three functions $x_0, x_1, x_2 \in \infseq{\Seq}$ such that $x_0(n), x_1(n), x_2(n)$ are pairwise distinct and not from $D_n$.
		Assume there were $(h_0^k,h_1^k,i^k)$ for $k \in 3$ such that
		\[
			\bigcap_{k \in 3} E(x_k, \dot{\text{e}}(h_0^k,h_1^k,i^k))
		\]
		is infinite.
		Because $x_k(n) \notin D_n$ we get that $E(x_k, \dot{\text{e}}(h_0^k,h_1^k,i^k)) \subseteq C(h_0^k,h_1^k,i^k)$.
		By almost disjointness of the $C(h_0,h_1)$ in fact all $(h_0^k,h_1^k,i^k)$ have to agree on the first two coordinates.
		Hence, by pigeonhole there are $k \neq l \in 3$ with $(h_0^k,h_1^k,i^k) = (h_0^l,h_1^l,i^l)$.
		But then $x_k$ and $x_l$ agree on $E(x_k, \dot{\text{e}}(h_0^k,h_1^k,i^k)) \cap E(x_l, \dot{\text{e}}(h_0^l,h_1^l,i^l))$, a contradiction.
	\end{proof}

	\noindent Summarizing everything we obtain the following result:
	
	\begin{corollary}
		$\calE_2$ is a Borel square but not cubic med family.
	\end{corollary}
	
	Next, we describe how to further generalize the construction to obtain for every $n \geq 1$ a Borel $n$-productive but not $(n+1)$-productive med family $\calE_n$.
	By the same discussion as before the natural indexing set for $\calE_n$ is $\calR_n := \underset{n\text{-times}}{\infseq{\Seq} \times \dots \times \infseq{\Seq}} \times n = (\infseq{\Seq})^n \times n$.
	
	\begin{theorem}\label{THM_BorelNProductive}
		There is a Borel $n$-productive but not $(n+1)$-productive med family $\calE_n$.
	\end{theorem}
	
	\begin{proof}
		Fix a computable bijection $\Phi_n:\calR_n \to \bairespace$ with computable inverse $\Psi_n$.
		We will modify the eventually different family
		\[
			\set{\text{e}(\Phi_n(\vec{h}, i))}{(\vec{h}, i) \in \calR_n}
		\]
		to be $n$-productive.
		Let $\ast(\vec{h})$ denote the standard coding of $n$-tuples of reals by reals, then we define the almost disjoint supports for copying
		\[
			B(\vec{h}) := B(\ast(\vec{h})) = \set{2 \langle \ast(\vec{h}) \restr k \rangle + 1}{k \in \omega}.
		\]
		Given $\vec{h} = (\enumzero{h}{n - 1}) \in (\infseq{\Seq})^n$, we will iteratively define a decreasing sequence of infinite subsets of $B(\vec{h})$
		\[
			X_0 \supseteq X_1 \supseteq \dots \supseteq X_n.
		\]	
		Set $X_0 := B(\vec{h})$ and assume $X_i$ has been defined.
		If $T_{X_i}(h_i)$ is infinite, we set $X_{i + 1} := T_{X_i}(h_i)$ and call $i$ a copyable coordinate for this requirement.
		Otherwise, let $X_{i+1}$ be the canonical strictly increasing chain in $X_i$, let $f_{h_i}$ be the corresponding union of $h_i(m_j)$ for $m_j \in X_{i + 1}$ and call $i$ a coherent coordinate with $f_{h_i}$ for this requirement.
		
		Now, let $S(\vec{h}) \subseteq n$ be the set of all copyable coordinates for this requirement.
		Then there is an $m(\vec{h}) \leq \sizeof{S(\vec{h})}$, a surjective function $\pi(\vec{h}):S(\vec{h}) \to m(\vec{h})$ and an infinite subset $C(\vec{h})$ of $X_n$ such that
		\begin{enumerate}
			\item for all $i,j \in S(\vec{h})$ with $\pi(\vec{h})(i) = \pi(\vec{h})(j)$ we have $h_i(n) = h_j(n)$ for all $n \in C(\vec{h})$,
			\item for all $i,j \in S(\vec{h})$ with $\pi(\vec{h})(i) \neq \pi(\vec{h})(j)$ we have $h_i(n) \neq h_j(n)$ for all $n \in C(\vec{h})$.
		\end{enumerate}
		It remains to discuss how to construct $\dot{\text{e}}(\vec{h}, i)$.
		Let $T(\vec{h}) := n \setminus S(\vec{h})$ be the set of all coherent coordinates.
		Then there is an injective function $\psi(\vec{h}):m(\vec{h}) \to n$ such that for all $i \in m(\vec{h})$ and $j \in T(\vec{h})$ we have
		\[
			\Phi_n(\vec{h}, \psi(\vec{h})(i)) \neq f_{h_j}.
		\]
		Finally, we define for $i \in m(\vec{h})$ the function
		\[
			\dot{\text{e}}(\vec{h}, \psi(\vec{h})(i))(k) :=
			\begin{cases}
				h_j(k) & \text{if } k \in C(\vec{h}),\\
				\text{e}(\Phi_n(\vec{h}, \psi(\vec{h})(i)))(k) & \text{otherwise},
			\end{cases}
		\]
		where $j$ is some (or any by (1)) member of $(\pi(\vec{h}))^{-1}(i)$.
		For any other $i \in n \setminus \ran(\psi(\vec{h}))$ we just set $\dot{\text{e}}(\vec{h}, i) := \text{e}(\Phi_n(\vec{h}, i))$ and define
		\[
			\calE_n := \set{\dot{\text{e}}(\vec{h}, i)}{(\vec{h}, i) \in \calR_n}.
		\]
		Given $(\vec{h},i),(\vec{g}, j) \in \calR_n$, since the starting family is e.d., the copy supports are almost disjoint and by the definition of copyable, by the same argument as before we may assume $\vec{h} = \vec{g}$, $i \neq j$ and $i,j \in \ran(\psi(\vec{h}))$, when checking eventual difference of $\dot{\text{e}}(\vec{h}, i)$ and $\dot{\text{e}}(\vec{g}, j)$.
		But then by (2) for $a, b \in S(\vec{h})$ with $(\psi(\vec{h}) \circ \pi(\vec{h}))(a) = i$ and $(\psi(\vec{h}) \circ \pi(\vec{h}))(b) = j$ we have
		\[
			h_a(k) \neq h_b(k) \text{ for all } k \in C(\vec{h}),
		\]
		which shows that for all $k \in C(\vec{h})$ we have
		\[
			\dot{\text{e}}(\vec{h}, i)(k) = h_a(k) \neq h_b(k) = \dot{\text{e}}(\vec{g}, j)(k).
		\]
		For $n$-productivity of $\calE_n$ let $\vec{h} \in (\infseq{\Seq})^n$.
		For $i \in S(\vec{h})$ we let $g_i := \dot{\text{e}}(\vec{h}, (\psi(\vec{h}) \circ \pi(\vec{h}))(i))$ and for $i \in T(\vec{h})$ we let $g_i := \dot{\text{e}}(\Psi_n(f_{h_i}))$, so that by the same argument as before we obtain
		\[
			E(h_i, g_i) \supseteq^* C(\vec{h}).
		\]
		But this implies
		\[
			\bigcap_{i \in n} E(h_i, g_i) \supseteq^* C(\vec{h})
		\]
		is infinite, verifying that $\calE_n$ is $n$-productive.
		Further, by a similar argument as Lemma~\ref{LEM_E2NotCubic} we have that $\calE_n$ is not $(n+1)$-productive, so it remains to argue that $\calE_n$ is Borel, and even arithmetic:
		\begin{enumerate}[$\circ$]
			\item By the previous discussion the copyable coordinates $S(\vec{h})$ and coherent coordinates $T(\vec{h})$ together with their witnesses $f_{h_i}$ are arithmetical in $\vec{h}$,
			\item In order to arithmetically find a canonical $\pi(\vec{h})$, $m(\vec{h})$ and $C(\vec{h})$ we can define an arithmetical in $\vec{h}$ equivalence relation $\equiv$ on $X_n$ by
			\[
				k \equiv l \quad \text{iff} \quad \forall i,j \in S(\vec{h}) \text{ we have } (h_i(k) = h_j(k) \iff h_i(l) = h_j(l))
			\]
			There are only finitely many equality patterns on the finite set $S(\vec{h})$.
			Thus, there are only finitely many equivalence classes for $\equiv$ on the infinite set $X_n$, so we let $C(\vec{h})$ be the first (in some canonical ordering) infinite such equivalence class.
			Now, define a second equivalence relation $\sim$ on $S(\vec{h})$, also arithmetical in $\vec{h}$:
			\[
				i \sim j \quad \text{iff} \quad \forall k \in C(\vec{h}) \text{ we have } h_i(k) = h_j(k).
			\]			
			 Let $m(\vec{h})$ be the number of equivalence classes of $\sim$ and $\pi(\vec{h}):S(\vec{h}) \to m(\vec{h})$ the canonical quotient map, which are both arithmetical in $\vec{h}$,
			\item There is also a canonical choice for $\psi(\vec{h}):m(\vec{h}) \to n$: There are at most $\sizeof{T(\vec{h})}$ many coordinates forbidden by the requirement that for all $i \in m(\vec{h})$ and $j \in T(\vec{h})$
			\[
				\Phi_n(\vec{h}, \psi(\vec{h})(i)) \neq f_{h_j}.
			\]
			But $n - \sizeof{T(\vec{h})} = \sizeof{S(\vec{h})} \geq m(\vec{h})$, so we can just define $\psi(\vec{h})$ by assigning the allowed coordinates in order, which is arithmetical in $\vec{h}$,
			\item Clearly, the definition of $\dot{\text{e}}(\vec{h}, i)$ is then arithmetical in $\vec{h}$ and $i$.
		\end{enumerate}
		As before, this shows that $\calE_n$ is arithmetic, because again the even coordinates may recover the parameters $(\vec{h}, i)$ of the construction.
	\end{proof}

	Finally, we may also obtain a Borel finitely productive med family $\calE_{<\omega}$.
	Finitely productive means $n$-productive for every $n \geq 1$, so the natural indexing set for $\calE_{<\omega}$ is the disjoint union of the $\calR_n$ given by
	\[
		 \calR_{<\omega} := \set{(n,\vec{h}, i)}{1 \leq n < \omega, \vec{h} \in (\infseq{\Seq})^n, i \in n}.
	\]
	We only sketch the proof as we basically just uniformly perform all previous constructions at the same time.

	\begin{theorem}\label{THM_BorelFinitelyProductive}
		There is a Borel finitely productive med family $\calE_{<\omega}$.
	\end{theorem}

	\begin{proof}
		Fix a computable bijection $\Phi_{<\omega}:\calR_{<\omega} \to \bairespace$ with computable inverse $\Psi_{<\omega}$.
		We will modify the eventually different family
		\[
			\set{\text{e}(\Phi_{<\omega}(n, \vec{h}, i))}{(n, \vec{h}, i) \in \calR_{<\omega}}
		\]
		to be finitely productive.
		We again define almost disjoint copy supports
		\[
			B(n, \vec{h}) := B(\ast(n, \vec{h})) = \set{2\langle\ast(n, \vec{h}) \restr k \rangle + 1}{k \in \omega},
		\]
		where $\ast(n, \vec{h})$ is a coding of pairs of a natural number $n$ and a finite sequence of reals $\vec{h}$ of length $n$.
		Given $(n, \vec{h}, i) \in \calR_{<\omega}$, where $\vec{h} = (\enumzero{h}{n-1})$, we define $\dot{\text{e}}(n, \vec{h}, i)$ by the same construction as in $\calE_n$ starting with the copy domain $X_0 := B(n, \vec{h})$ and set
		\[
			\calE_{<\omega} := \set{\dot{\text{e}}(n, \vec{h}, i)}{(n, \vec{h}, i) \in \calR_{<\omega}}.
		\]
		By the same arguments as before $\calE_{<\omega}$ is eventually different and $n$-productive for every $n \geq 1$, so that $\calE_{<\omega}$ is finitely productive.
		Then, $\calE_{<\omega}$ is Borel, because the construction of $\calE_n$ is arithmetical uniformly in $n$.
	\end{proof}
	
	\section{Closed productive families}
	
	In this final section we will show how to reduce the complexity of the productive families of the last section to obtain closed families.
	In \cite{Schrittesser_2016} Schrittesser already proved that this is possible for the Horowitz-Shelah med family $\calE$.
	Schrittesser proved that the complexity of the original family $\calE$ is arithmetic and then gave an inductive argument to reduce the complexity one step at a time to obtain a closed med family.
	Here, we give a more direct argument to directly construct such closed families, which may also be applied to the original Horowitz-Shelah family to obtain a closed med family.
	We use that every Borel subset of a Polish space is the continuous injective image of a closed subset of $\bairespace$ \cite[15.3]{Kechris_1995}
	
	\begin{lemma}
		Assume there is a Borel med family $\calF$.
		Then there is a closed med family $\calG$.
		Furthermore, $\calF$ is $n$-productive iff $\calG$ is $n$-productive.
	\end{lemma}
	
	\begin{proof}
		Assume $\calF$ is a Borel med family.
		Choose a closed $C \subseteq \bairespace$ and a continuous injection $\iota:C \to \bairespace$ such that $\iota[C] = \calF$.
		Thus, also
		\[
			H := \set{(\iota(c), c)}{c \in C}
		\]
		is a closed subset of $\bairespace \times \bairespace$ and $\calF$ is the projection of $H$ onto the first component.
		Now, we can use the standard trick for coding with med families (see \cite{Schrittesser_2016} or \cite{FischerSchrittesser_2021}):
		
		Fix an injective coding $\langle \cdot, \cdot \rangle$ of the pairs of sequences of the same finite length by natural numbers such that $0 \notin \ran(\langle \cdot, \cdot \rangle)$.
		Now, define a continuous map $\chi:\bairespace \times \bairespace \to \bairespace$ by
		\begin{align*}
			\chi(f,c)(2n) &:= f(n)\\
			\chi(f,c)(2n+1) &:= \langle f \restr n, c \restr n \rangle
		\end{align*}
		Finally, we set $\calG := \chi(H)$.
		First, we show that $\chi$ is a closed map, so let $A\subseteq \bairespace \times \bairespace$ be closed and suppose $(\chi(f_k,c_k))_{k \in \omega}$ converges to some $g\in\bairespace$, where $(f_k,c_k)\in A$ for all $k$.
		Then for each $n$ the sequence
		\[
			\chi(f_k, c_k)(2n+1) = \langle f_k \restr n, c_k \restr n \rangle
		\]
		is eventually constant. Since the coding $\langle\cdot,\cdot\rangle$ is injective, it follows that for each $n$ both sequences
		\[
			f_k\restr n \quad \text{and} \quad c_k\restr n
		\]
		are eventually constant.
		Hence $(f_k,c_k)$ converges to some $(f,c)\in\bairespace\times\bairespace$.
		Since $A$ is closed we have $(f,c)\in A$.
		Moreover, by construction
		\[
			g(2n) = f(n) \quad \text{and} \quad g(2n+1) = \langle f \restr n, c \restr n \rangle
		\]
		for every $n$.
		Hence, $g = \chi(f,c) \in \chi[A]$.
		Therefore $\chi[A]$ is closed.
		In particular, since $H$ is closed, $\calG=\chi[H]$ is closed.
		
		It is also clear that the construction ensures that $\calG$ is ed, so it remains to show that $\calF$ is $n$-productive iff $\calG$ is $n$-productive.
		So assume $\calF$ is $n$-productive and let $\enumone{y}{n} \in \bairespace$.
		Define $\enumone{x}{n} \in \bairespace$ by
		\[
			x_i(k) := y_i(2k).
		\]
		Since $\calF$ is $n$-productive, there are $\enumone{f}{n} \in \calF$ such that $\bigcap_{i=1}^n E(x_i, f_i)$ is infinite.
		For each $1 \leq i \leq n$ let $c_i$ be the unique element of $C$ such that $(f_i,c_i)\in H$.
		Note that for every $i$ and every $k$ we have that $x_i(k) = f_i(k)$ implies $y_i(2k)=\chi(f_i,c_i)(2k)$.
		Hence $\bigcap_{i=1}^n E(y_i, \chi(f_i, c_i))$ is infinite.
		
		Conversely, assume $\calG$ is $n$-productive and let $\enumone{x}{n} \in \bairespace$.
		Define $\enumone{y}{n}\in\bairespace$ by
		\[
			y_i(2k) := x_i(k) \quad \text{and} \quad y_i(2k+1) := 0.
		\]
		By $n$-productivity of $\calG$, there are $(f_i,c_i) \in H$ such that $\bigcap_{i = 1}^n E(y_i,\chi(f_i, c_i))$ is infinite.
		Since $0\notin \ran(\langle \cdot,\cdot\rangle)$, we have
		\[
			y_i(2k+1) \neq \chi(f_i, c_i)(2k+1)
		\]
		for all $1\leq i\leq n$ and all $k\in\omega$.
		Hence, $\bigcap_{i=1}^n E(y_i, \chi(f_i,c_i))$ contains infinitely many even numbers, so that $\bigcap_{i=1}^n E(x_i,f_i)$ is infinite.
		Therefore, $\calF$ is $n$-productive if and only if $\calG$ is $n$-productive.
	\end{proof}
	
	\begin{corollary} \label{COR_ClosedProductive}
		For every $n \geq 1$ there is a closed $n$-productive but not $(n+1)$-productive med family and there is a closed finitely productive med family.
	\end{corollary}

	\begin{proof}
		Apply the previous lemma to the families $\calE_n$ and $\calE_{<\omega}$ of the last section.
	\end{proof}
	
	Finally, there are some natural questions how complexity restrictions for med family $\calF$ restrict how the associated increasing chain of ideals
	\[
		\calI_0(\calF) \subseteq \calI_0(\calF^2) \subseteq \calI_0(\calF^3) \subseteq \dots
	\]
	looks like.
	In terms of these chains, we showed that for every $n \geq 1$, there is a closed med family $\calF$ such that $\calI_0(\calF^m)$ is proper exactly up to $m = n$, and there is closed med family such that $\calI_0(\calF^n)$ is proper for all $n \in \omega$.
	Now, since Van Douwen family cannot be analytic, by Lemma~\ref{PROP_RestrictionToVD} we at least know that $\calI_0(\calF)$ has to be tall for an analytic $\calF$.
	
	\begin{question}
		What other restrictions does the associated chain of an analytic maximal eventually different family have? Do we have analytic/Borel/closed $n$-$\calI_0$-productive families for $n > 1$? Or even finitely $\calI_0$-productive such families?
	\end{question}

	Note that there may be $\Pi^1_1$ Van Douwen families and thus $\Pi^1_1$ finitely $\calI_0$-productive families by the results of Millhouse and the author in \cite{MillhouseSchembecker_2025}.
	Lastly, we show that the family $\calE_{<\omega}$ is not even $2$-$\calI_0$-productive, in fact we have the following stronger observation:
	
	\begin{lemma}\label{LEM_StrictlyIncreasingI0}
		The associated chain of ideals of $\calE_{<\omega}$ is strictly increasing, i.e.\
		\[
			\calI_0(\calE_{<\omega}) \subsetneq \calI_0((\calE_{<\omega})^2) \subsetneq \calI_0((\calE_{<\omega})^3) \subsetneq \dots
		\]
	\end{lemma}

	\begin{proof}
		For $n \geq 1$ we define
		\[
			A_n := \bigcup_{\smash{\sizeof{\vec{h}} = n}} B(n, \vec{h}).
		\]
		Note that the $A_n$'s are pairwise almost disjoint and we show that
		\[
			A_n \in \calI_0((\calE_{<\omega})^{n+1}) \setminus \calI_0((\calE_{<\omega})^n).
		\]
		Towards $A_n \notin \calI_0((\calE_{<\omega})^{n})$, let $\enumzero{\hat{h}}{n-1}:A_n \to \Seq$ be infinite partial functions.
		Extend every $\hat{h}_i$ to $h_i:\omega \to \Seq$ arbitrarily.
		Then by construction of $\calE_{<\omega}$ there are $\enumzero{g}{n-1} \in \calE_{<\omega}$ such that
		\[
			\bigcap_{i \in n} E(g_i, f_i) \supseteq^* C(n, \vec{h}).
		\]
		But $C(n, \vec{h}) \subseteq B(n, \vec{h}) \subseteq A_n$, so that also $\bigcap_{i \in n} E(g_i, \hat{h}_i)$ is infinite.
		This shows that $(\calE_{<\omega})^n \restr A_n$ is maximal, i.e.\ $A_n \notin \calI_0((\calE_{<\omega})^{n})$.
		
		Towards $A_n \in \calI_0((\calE_{<\omega})^{n+1})$, we choose $\enumzero{x}{n}: A_n \to \Seq$ so that the $x_i(k)$'s are pairwise distinct and $x_i(k) \notin D_k$.
		Towards a contradiction, assume that there were $\enumzero{g}{n} \in \calE_{<\omega}$ with
		\[
			\bigcap_{i\in n+1} E(x_i, f_i)
		\]
		infinite.
		For each $i \in n + 1$ let $(n_i, \vec{h}_i, l_i)$ be the parameters recovered from the even coordinates of $g_i$.
		Since $x_i(k) \notin D_k$ we have that $E(x_i,f_i)$ must be a subset of the copy support $B(n_i, \vec{h}_i) \subseteq A_{n_i}$.
		But $\dom(x_i) = A_n$ and the $A_n$ are almost disjoint, which implies that $n_i = n$ for all $i \in n + 1$.
		But distinct copy supports are almost disjoint, so we must have that $(n_i, \vec{h}_i, l_i) = (n, \vec{h}, l_i)$ for some common $\vec{h} \in (\infseq{\Seq})^n$.
		By pigeonhole, there are $i \neq j$ with $(n_i, \vec{h}_i, l_i) = (n_j, \vec{h}_j, l_j)$, so $g_i = g_j$.
		But this contradicts that $x_i(k) \neq x_j(k)$ for all $k \in A_n$.
	\end{proof}

	\bibliographystyle{plain}
	\bibliography{refs}
	
\end{document}